\documentclass[11pt]{article}
\usepackage{amsthm}
\usepackage{graphicx, amssymb, amstext, amsmath, pdfsync, tikz, amscd}
\usepackage{tikz}
\usepackage{tikz}
\usetikzlibrary{calc}
\usetikzlibrary{matrix,arrows}
\usepackage{verbatim}
\usepackage{tocloft}
\usepackage{setspace}

\usepackage[round, sort, numbers]{natbib}
\setcitestyle{square}

\usepackage[english]{babel}
\usepackage[none]{hyphenat}
\usepackage[tmargin=1in,bmargin=1in,lmargin=.5in,rmargin=.5in]{geometry}

\usepackage{color}
\usepackage[colorlinks, pagebackref]{hyperref}
\hypersetup{
  colorlinks=true,
  citecolor=green,
  linkcolor=blue,
  urlcolor=blue}

\newtheorem{introtheorem}{Theorem}[]
\newtheorem{introconjecture}{Conjecture}[]
\newtheorem{introdefinition}{Definition}[]

\newtheorem{theorem}{Theorem}[subsection]
\newtheorem{lemma}[theorem]{Lemma}
\newtheorem{proposition}[theorem]{Proposition}
\newtheorem{corollary}[theorem]{Corollary}
\newtheorem{conjecture}[theorem]{Conjecture}

\theoremstyle{definition}
\newtheorem{definition}[theorem]{Definition}

\newtheorem{remark}[theorem]{Remark}
\newtheorem{example}[theorem]{Example}

\begin{document}
\title{Action of Correspondences on Filtrations on Cohomology and 0-cycles of Abelian Varieties} 
\author{Rakesh R. Pawar}

\date{}
\maketitle

\abstract  We prove that, given a symmetrically distinguished correspondence of a suitable complex abelian variety (which include any abelian variety of dimension atmost 5, powers of complex elliptic curves, etc.) 
which vanishes as a morphism on a certain quotient of its middle singular cohomology, then it vanishes as a morphism on the deepest part of a particular filtration on the Chow group of 0-cycles of the abelian variety. As a consequence, we prove that given an automorphism of such an abelian variety, which acts as the identity on a certain quotient of its middle singular cohomology, then it acts as the identity on the deepest part of this filtration on the Chow group of 0-cycles of the abelian variety. As an application, we prove that for the Generalized Kummer variety associated to a complex abelian surface and the automorphism induced from a symplectic automorphism of the complex abelian surface, the automorphism of the Generalized Kummer variety acts as the identity on a certain subgroup of its Chow group of 0-cycles. 

\tableofcontents
\begin{section}{Introduction}
Given a smooth projective variety $X$ over an algebraically closed field $k\subset\mathbf{C}$, we can associate two important invariants, the Chow groups and singular cohomology groups, both indexed by non-negative integers. These two invariants are related by cycle class homomorphisms from Chow groups to even degree cohomology groups. Further, the kernel of the cycle class map is related to a Hodge structure on the singular cohomology groups by the Abel-Jacobi map. D. Mumford showed in~\cite{mumford68} that, if $X$ is a complex surface which admits a non-zero holomorphic 2-form (i.e. $H^0(X,\Omega^2_X)\neq0$), then $CH_0(X)$ (the Chow group of 0-dimensional cycles) is ``infinite dimensional". This result suggested that the singular cohomology groups, or rather the Hodge structure on the cohomology groups, dictates the structure of the Chow groups. 

Conjectural formulation of such a relation was first initiated by S. Bloch (See~\cite[Conjecture 1.8]{bl2}). This has been vastly generalized into a finer conjecture known as the Bloch-Beilinson conjecture~\cite{j}, which says that  
\begin{introconjecture}[Bloch-Beilinson]
If $X$ is a smooth projective variety of dimension $d$ over $\mathbf{C}$, then for each $k\geq0$, there exists a decreasing filtration $G^{\bullet}CH^k(X)$ on the Chow groups with rational coefficients satisfying:
\begin{enumerate}
\item (Finiteness) $G^{k+1}CH^k(X)=0$
\item (Functoriality) The filtration $G^i$ is stable under correspondences: if $Y$ is a smooth projective variety over $\mathbf{C}$ and $\Gamma\in CH^l(X\times Y)$, then the maps 
\begin{center}
$\Gamma_{\ast}:CH^k(X)\to CH^{k+l-d}(Y)$
\end{center}
satisfy
\begin{center}
$\Gamma_{\ast}G^iCH^k(X) \subset G^iCH^{k+l-d}(Y).$
\end{center}
\item The induced map 
\begin{center}
$gr^i_G\Gamma_{\ast}: gr^i_G CH^k(X)\to gr^i_G CH^{k+l-d}(Y)$
\end{center}
vanishes if the map 
\begin{center}$[\Gamma]_{\ast}:H^{2k-i}(X,\mathbf{Q})\to H^{2k+2l-2d-i}(Y, \mathbf{Q})$
\end{center}
vanishes on $H^{r,s}(X)$ for $s\leq k-i.$
\end{enumerate}
 \end{introconjecture}
In particular, for $k=d=$ dim $X$, the conjecture predicts for the Chow group of 0-cycles with $\mathbf{Q}$-coefficients $CH^d(X)=CH_0(X)$ the following:
\begin{introconjecture}[Bloch-Beilinson for 0-cycles]
\label{bb_0}
There exists a decreasing filtration $G^{\bullet}CH_0(X)$ such that $G^{d+1}CH_0(X)=0$, stable under suitable correspondences and satisfying:\\
 For a correspondence $\Gamma\in CH^{d}(X\times X)$, the map 
\begin{center}
 $gr^i_G\Gamma_{\ast}:gr^i_G CH_0(X)\to gr^i_G CH_0(X)$
\end{center}
is 0 if 
\begin{center}
$[\Gamma]_{\ast}:H^{d-i}(X,\Omega_X^{d})\to H^{d-i}(X, \Omega_X^{d})$\end{center}
is 0.


\end{introconjecture}
Our goal here is to study a particular natural filtration $G^{\bullet}$ on $CH_0(X)$ for $X$ an abelian variety over an algebraically closed field $k$ of char 0. This filtration was first studied by S. Bloch in~\cite{bl1}.
\begin{introdefinition}\label{filtration0cycles}
Let $Pic(X)$ (and $Pic^0(X)$) be the group of divisors  (and divisors algebraically equivalent to 0, resp.)on $X$, modulo rational equivalence with $\mathbf{Q}$-coefficients. Now the intersection of divisors gives a decreasing filtration $G^{\bullet}$ on the Chow group of 0-cycles, $CH_0(X)$ as follows:
\begin{center}
  $G^i:=Pic^0(X)^{\cap i}\cap Pic(X)^{\cap (d-i)}$, for $i\geq 0$ 
  \end{center}
  where 
  \begin{center}
$Pic^0(X)^{\cap i}\cap Pic(X)^{\cap (d-i)}$:=Im $(Pic^0(X)^{\otimes i}\otimes Pic(X)^{\otimes (d-i)}\xrightarrow{(-)\cap(-)\cap\cdots\cap(-)} CH_0(X))$
  \end{center}
 and $\cap$ is the intersection product on cycles on $X$ as defined in \cite[Ch. 8]{fulton}. 
  \end{introdefinition}
  One expects that, this filtration satisfies the Bloch-Beilinson Conjecture \ref{bb_0} for the Chow group of 0-cycles on $X$ with rational coefficients. In particular, if $f: X\to X$ is an algebraic automorphism of $X$, $\Gamma=\Gamma_f-\Delta_X$, where $\Gamma_f$ is the graph of $f$ and $\Delta_X$ is the diagonal of $X$, the Conjecture~\ref{bb_0} predicts the following:

\begin{introconjecture}
\label{mq}
Let $X$ be an abelian variety of dimension $d$ over 
$\mathbf{C}.$ 
Suppose $f :X \to X$ is an automorphism of $X$ such that the induced morphsim $f_{\ast}: H^{0}(X,\Omega_X^d) \to H^{0}(X,\Omega_X^d)$ is the identity (abbreviated as Id). Then $f_{*}=identity: G^{d} \to  G^{d}$ induced by the restriction of $f_{\ast}: CH_0(X)\to CH_0(X)$.
\end{introconjecture}
This conjecture is difficult to answer in general. We will reformulate this conjecture to Conjecture \ref{mq1} which is more tractable in view of the recent results of C. Vial~\cite{v}. For this let us consider more closely the condition that $f_{\ast}: H^{0}(X,\Omega_X^d) \to H^{0}(X,\Omega_X^d)$ is the identity. For this discussion, $X$ can be any smooth projective variety of dimension $d$ over $\mathbf{C}$. Let $\Gamma_f$:= the graph of $f \in CH^d(X\times X)$. Then $[\Delta_X-\Gamma_f]_{\ast}$ is a $\mathbf{Q}$-linear map on the singular cohomology group $H^d(X):=H^d(X,\mathbf{Q})$ which after tensoring with $\mathbf{C}$ is 0 on $H^{0}(X,\Omega_X^d).$ Note that $[\Delta_X-\Gamma_f]_{\ast}:H^d(X)\to H^d(X)$ is a morphism of $\mathbf{Q}$-Hodge structures, hence ker$([\Delta_X-\Gamma_f]_{\ast})$ is a $\mathbf{Q}$-sub Hodge structure of $H^d(X)$, which contains $H^{0}(X,\Omega_X^d)$ after tensoring with $\mathbf{C}.$ Therefore,  ker$([\Delta_X-\Gamma_f]_{\ast})$ contains the smallest $\mathbf{Q}$-sub Hodge structure of $H^d(X)$ which contains $H^{0}(X,\Omega_X^d)$ after tensoring with $\mathbf{C}$, which we will write as $H^d(X)_{tr}$ [for transcendental cohomology]. Thus, we can replace the assumption in the Conjecture~\ref{mq} by the condition that $f_{\ast}: H^d(X)_{tr} \to H^d(X)_{tr}$ is the identity.  

In terms of the \textbf{Coniveau filtration} $N^{\bullet}H^d(X)$ on $H^d(X)$ defined in Definition~\ref{con} and the Hodge filtration $F^pH^d(X, \mathbf{C})$ on $H^d(X,\mathbf{C})$ defined in Definition~\ref{hodgefil}, 
\begin{center}
 $N^pH^d(X)\subseteq F^pH^d(X, \mathbf{C})\cap H^d(X)$, for $p\geq0$.
 \end{center}
 We recall the Grothendieck's Generalized Hodge Conjecture here. We say that $X$ satisfies \textbf{GHC($X, p, d$)}, if
the maximal $\mathbf{Q}$-Hodge structure contained in $F^pH^d(X, \mathbf{C})\cap H^d(X)$ is $N^pH^d(X)$.\\ 
Fixing a polarization on $H^d(X)$ using the intersection pairing, we have the orthogonal complement $H^d(X)_{tr}^{\perp}$ of $H^d(X)_{tr}$ in $H^d(X)$, which is a sub $\mathbf{Q}$-Hodge structure contained in $F^1H^d(X, \mathbf{C})\cap H^d(X)$. 
By the earlier discussion, $H^d(X)_{tr}$ is contained in $N^1H^d(X)^{\perp}$ as $N^1H^d(X)^{\perp}$ is a $\mathbf{Q}$-sub Hodge structure containing $H^{0}(X,\Omega_X^d)$ in the Hodge decomposition. Hence $N^1H^d(X)\subseteq H^d(X)_{tr}^{\perp}$. Thus, we have a surjective homomorphism
\begin{center}
$\dfrac{H^d(X)}{N^1H^d(X)}\twoheadrightarrow\dfrac{H^d(X)}{H^d(X)_{tr}^{\perp}}\simeq H^d(X)_{tr}$
\end{center}
which will be an isomorphism assuming GHC($X, 1, d$). Examples of smooth projective varieties satisfying Grothendieck's generalized Hodge Conjecture will be listed in the Section~\ref{list}.

 Next, we consider the \textbf{Niveau filtration} $\widetilde{N}^{\bullet}$ on the (co)homology of $X$ in the next section (see Definition~\ref{nev}). Using the properties of Niveau and Coniveau filtration viz., Eq.~(\ref{inclusion}), we have surjections
 \begin{center}
 $\dfrac{H^d(X)}{\widetilde{N}^1H^d(X)}\twoheadrightarrow\dfrac{H^d(X)}{N^1H^d(X)}\twoheadrightarrow\dfrac{H^d(X)}{H^d(X)_{tr}^{\perp}}\simeq H^d(X)_{tr},$
 \end{center}
 where the first arrow is an isomorphism if the Lefschetz standard conjecture holds for all smooth projective varieties of dimension $d-2$ (see the proof of \cite[Proposition 1.1]{v}) and the second arrow is an isomorphism if GHC($X, 1, d$) holds true. 
 
 With this in hindsight, we will modify the Conjecture~\ref{mq}, as follows:
\begin{introconjecture}
\label{mq1}
Let $X$ be an abelian variety of dimension $d$ over 
$\mathbf{C}.$ Suppose $f :X \to X$ is an automorphism of $X$ such that the induced morphism $f_{\ast}: \dfrac{H^d(X)}{\widetilde{N}^1H^d(X)} \to \dfrac{H^d(X)}{\widetilde{N}^1H^d(X)}$ is the Id. Then $f_{*}=Id: G^{d} \to  G^{d}$ (as in Definition~\ref{filtration0cycles}) induced by the restriction of $f_{\ast}: CH_0(X)\to CH_0(X)$. 
\end{introconjecture}
By the discussion above, clearly the assumption on $f$ in the Conjecture~\ref{mq1} implies the assumption on $f$ in the Conjecture \ref{mq}. Conversely, assume that $f$ satisfies the assumption as in the Conjecture~\ref{mq}, then if we assume that GHC($X, 1, d$) holds, and also the Lefschetz standard conjecture holds for all smooth projective varieties of dimension $d-2$, then $f$ satisfies the assumption in Conjecture~\ref{mq1}. 
\paragraph{Statements of the results}
In this paper, we propose to answer Conjecture~\ref{mq} and~\ref{mq1} for suitable abelian varieties $X$, which follows from the following more general theorem.
The main theorem in this paper is:
\begin{introtheorem}[See Theorem~\ref{main}]
\label{intromain}
Let $X$ be a complex abelian variety satisfying the assumption ($\ast$) as in the Section~\ref{refinedCK}.
Suppose $\Gamma\in CH^d(X\times X)$ is a correspondence from $X$ to $X$ which is symmetrically distinguished. Let 
\begin{center}
$[\Gamma]_{\ast}: \dfrac{H^d(X)}{\widetilde{N}^1H^d(X)} \to \dfrac{H^d(X)}{\widetilde{N}^1H^d(X)}$
\end{center} be 0 homomorphism.
Then $[\Gamma]_{*}=0: \cap_{i=1}^{d} Pic^0(X) \to  \cap_{i=1}^{d} Pic^0(X)$ induced by the restriction of $[\Gamma]_{*}:CH_0(X)\to CH_0(X)$.
\end{introtheorem}
We note that the assumption ($\ast$) in the Theorem~\ref{intromain} holds for many abelian varieties, e.g. listed in the Section~\ref{list}. Also symmetrically distinguished cycles on abelian varieties were defined by P. O'Sullivan in~\cite[Definition 6.2.1]{o}. 
 \\
As a consequence, we prove,
\begin{introtheorem}[See Corollary \ref{mc}]
\label{mt}
Let 
$X$ be a complex abelian variety of dimension d satisfying the assumption ($\ast$) as in the Section~\ref{refinedCK}. Then $X$ satisfies the Conjecture~\ref{mq1}. 
\end{introtheorem}
Theorem~\ref{mt} can be applied to many abelian varieties listed in the Section~\ref{list}, including arbitrary abelian varieties of dimension atmost 5, arbitrary products of elliptic curves, etc.\\ 
We deduce Conjecture~\ref{mq} from Theorem~\ref{mt}, for suitable abelian varieties. 
\begin{introtheorem}\label{mtghc}
Let $X$ be a complex abelian variety of dimension $d$ satisfying the assumption ($\ast$) as above, $N^1H^d(X)=\widetilde{N}^1H^d(X)$ and furthermore GHC($X, 1, d$) holds true. 
Then $X$ satisfies the Conjecture~\ref{mq}. 
\end{introtheorem}
Theorem~\ref{mtghc} follows from Theorem~\ref{mt} and the discussion following the Conjecture~\ref{mq}. 
We note that the Theorem~\ref{mtghc} can be applied to abelian varieties listed
in the Section~\ref{list}, for which the assumption GHC($X, 1, d$) holds true, which include among others:
\begin{enumerate}
\item abelian varieties of dimension $\leq2$,
\item abelian threefolds for which GHC($X, 1, 3$) holds true,
\item arbitrary products of elliptic curves. 
\end{enumerate}

Finally, we discuss various applications of Theorem~\ref{mt}.  
We consider the
Kummer surfaces and in general \textit{Generalized Kummer varieties} and their Chow groups of 0-cycles. 

D. Huybrechts~\cite{h} and C. Voisin~\cite{vo1} have proved the following theorem:
\begin{introtheorem}
\label{introhv}[Huybrechts-Voisin]
Let $f$ be an automorphism of a projective K3 surface $X$ of finite order which is symplectic (i.e. which acts as the identity on $H^{0}(X,\Omega_X^2))$, then $f_{\ast}$ acts as the identity on $CH_0(X)$.
\end{introtheorem}
Motivated by this result, we prove a similar theorem for the Kummer surface $Km(X)$ associated to a complex abelian surface $X$, but for automorphisms (possibly of infinite order) induced from the complex abelian surface $X$. 

More precisely, we prove the following:
\begin{introtheorem}[See Theorem~\ref{km}]
\label{introkm}
Let $g$ be an automorphism of the Kummer surface $Km(X)$ associated to the abelian surface $X$, which is induced from an automorphism $f$ of $X$ as an abelian variety. Further assume that $f$ is a symplectic automorphism of $X$. Then $g$ acts as the identity on $CH_0(Km(X)).$
\end{introtheorem}
In view of the Theorem~\ref{introhv}, the Theorem~\ref{introkm} deals with automorphisms of the Kummer surface which are possibly of infinite order. 

We further consider higher dimensional generalization of this setup, where $n=1$ is the earlier case.
Given a complex abelian surface, we consider natural hyperk\"ahler varieties $K_n$ associated to the abelian surface for $n\geq 2$ , introduced by A. Beauville~\cite{b83}, called \textit{Generalized Kummer varieties}. 

We first study the structure of their Chow group of 0-cycles with $\mathbf{Q}$-coefficients.
%

 One has a Beauville type decomposition of the 0-cycles of $K_n$ (see Eq.~(\ref{0cyclesonkummer})), 
 \begin{center}
$\displaystyle CH_0(K_{n})=\bigoplus_{s=0}^{2n} CH_0(K_{n})_s$
\end{center}

As a consequence, we prove the following theorem:
\begin{introtheorem}[see Theorem~\ref{gkm}]
\label{gkmi}
Let $X$ be a complex abelian surface. Let $K_n$ be the ${n^{th}}$ Generalized Kummer variety associated to $X.$
Suppose $f_n$ is the automorphism of $K_n$ induced from an automorphism $f$ of $X$ such that $f$ is symplectic on $X$ (i.e. which acts as the identity on $H^0(X,\Omega_X^2)$). Then $f_n$ is symplectic and $f_{n}$ acts as the identity on the subgroup $CH_0(K_n)_{2n}.$
\end{introtheorem} 
Throughout the paper, we will be using the following notations:
 \paragraph{Notations}
$X$ will denote a smooth projective variety of pure dimension $d$ over an algebraically closed field $k$, a subfield of the field of complex numbers $\mathbf{C}$ (in particular of characteristic 0), unless otherwise specified. 
 The Chow group of $X$ of codimension $i$-cycles with $\mathbf{Q}$ coefficients is denoted by $CH^i(X)$. If $X$ is an equi-dimensional variety of dimension $d$, $CH_i(X):=CH^{d-i}(X)$. 
 $CH_{num}^{i}(X)$ is $CH^i(X)$ modulo numerical equivalence.\\
\end{section}

\begin{section}{Preliminaries}\label{maintheorem}
\subsection{Filtrations on Cohomology and their Properties}
Given a smooth projective variety $X$ of dimension $d$ over an algebraically closed field $k\subset\mathbf{C}$, we associate the singular cohomology groups \begin{center}
$H^i(X):=H^i(X(\mathbf{C}), \mathbf{Q}) $ for $0\leq i\leq 2d,$
\end{center}
and Borel-Moore homology groups 
\begin{center}
$H_i(X):=H_i(X(\mathbf{C}), \mathbf{Q}) $ for $0\leq i\leq 2d.$
\end{center}

We will identify the Borel-Moore homology groups $H_i(X)$ with the singular cohomology groups $H^{2d-i}(X):=H^{2d-i}(X(\mathbf{C}), \mathbf{Q})$ via Poincare duality. We will consider $\mathbf{Q}$-Hodge structure on $H^i(X)$ 
, where 
\begin{center}
$H^{p,q}(X):=H^{q}(X, \Omega_X^{p})$ for $p, q\geq 0, p+q=i$.
\end{center}
\begin{definition}\label{hodgefil}
The Hodge decomposition induces a decreasing filtration, known as the \textbf{Hodge filtration} on $H^i(X, \mathbf{C})$, given by 
\begin{center}
$F^jH^i(X, \mathbf{C})= \oplus_{a\geq j} H^{a,i-a}(X)$ for $j\geq 0$,
\end{center}
where $H^i(X, \mathbf{C}):=H^i(X)\otimes_{\mathbf{Q}} \mathbf{C}.$
\end{definition}
Here we recall two filtrations on the singular cohomology groups, following the exposition in~\cite{v}.
The coniveau filtration on the (co) homology groups has been studied by many authors including Bloch-Ogus~\cite{BO}, Jannsen~\cite{j}. 
\begin{definition}
\label{con}
The \textbf{Coniveau (or arithmetic) filtration} on $H^i(X)$ is given by 
\begin{align*}
N^jH^i(X):&=\sum_{Z\subset X} \text{Ker} ( H^{i}(X)\to H^i(X-Z))\\
                 &=\sum_{Z\subset X} \text{Im} ( H^{i}_Z(X)\to H^i(X))\\
                 &= \sum_{f:Y\to X} \text{Im} (f_{\ast}: H^{i-2j}(Y)\to H^i(X))
\end{align*}
where $Z$ ranges over all closed subschemes of $X$ of codimension atleast $j$ and $f:Y\to X$ varies over all morphisms $f$ from smooth projective varieties $Y$ of dimension atmost $\dim X-j$ to $X$.
\end{definition}


There is yet another filtration called Niveau filtration which was implicit in~\cite{sch93}. Also the Niveau filtration was studied along with other filtrations on the homology by Friedlander-Mazur~\cite{frma}, where it was called Correspondence filtration. The Niveau filtration is more naturally defined on homology groups rather than cohomology groups.
\begin{definition}
\label{nev}
 The  \textbf{Niveau (or correspondence) filtration} on $H_i(X)$ can be defined as 
\begin{center}
$\widetilde{N}^{j}H_i(X):= \sum Im (\Gamma_{*}: H_{i-2j'}(W)\to H_i(X))$
\end{center}
where the sum ranges over all integers $j' \geq j$, all smooth projective varieties $W$ and over all correspondences $\Gamma\in CH^{d-j'}(W\times X)$. 

\end{definition}

Actually the condition $j'\geq j$ and flexibility over dimension of $W$ is not needed as if $j'> j$ then replace $W$ by $\mathbf{P}^{j'-j}\times W$ and $\Gamma$ by ${0}\times \Gamma \in CH^{d-j}(\mathbf{P}^{j'-j}\times W\times X)$. Thus can take $j'=j$ in the definition of niveau filtration. Further is is also possible to restrict to $W$ of dimension $i-2j$. If dim $W> i-2j$, then any smooth linear section $u: W'\to W$ of dimension $i-2j$ induces a surjection $H_{i-2j}(W')\to H_{i-2j}(W)$. Replace $W$ by $W'$ and $\Gamma$ by $\Gamma\circ u$. If dim $W< i-2j$, then replace $W$ by $\mathbf{P}^m\times W$ where $m=i-2j-\dim W$ and $\Gamma$ by $\mathbf{P}^m\times \Gamma\in CH^{d-j}(\mathbf{P}^m\times W\times X)$.
Therefore we have 
\begin{center}
$\widetilde{N}^{j}H_i(X)= \sum Im (\Gamma_{*}: H_{i-2j}(W)\to H_i(X))$
\end{center}
where the sum runs over all smooth projective varieties $W$ of dimension $i-2j$ and all correspondences $\Gamma\in CH^{d-j}(W\times X)=CH_{i-j}(W\times X).$ 
Because $H_i(X)$ is a finite dimensional $\mathbf{Q}$-vector space, we can find finitely many $W_s$ such that $\Gamma_s$ is a correspondence from $W_s$ to $X$ and take $W_{i,j} = W_1\sqcup\cdots\sqcup W_r$ and $\Gamma_{i,j}=\sqcup_s(\Gamma_s)$ satisfying $W_{i,j} $ is a smooth projective variety of pure dimension $i-2j$ and the correspondence $\Gamma_{i,j}\in CH_{i-j}(W_{i,j}\times X)$ such that 
\begin{equation}\label{nivsimple}
\widetilde{N}^jH_i(X)= (\Gamma_{i,j})_{*}H_{i-2j}(W_{i,j}).
\end{equation}
One defines the \textbf{Niveau filtration on Cohomology groups} as follows
\begin{align}
\widetilde{N}^jH^i(X):=\widetilde{N}^{j-i+d}H_{2d-i}(X).
\end{align}

These two filtrations are related to each other as follows:
\begin{enumerate}
\item $N^{\bullet}$ and $\widetilde{N}^{\bullet}$ are decreasing filtrations such that
\begin{align}
H^i(X)&=N^0H^i(X)\supseteq N^1H^i(X)\supseteq\cdots\supseteq N^{\lfloor i/2\rfloor} H^i(X)\supseteq N^{\lfloor i/2\rfloor+1} H^i(X)=0 \\ 
H^i(X)&=\widetilde{N}^0 H^i(X)\supseteq\widetilde{N}^1H^i(X)\supseteq\cdots\supseteq\widetilde{N}^{\lfloor i/2\rfloor} H^i(X)\supseteq \widetilde{N}^{\lfloor i/2\rfloor+1} H^i(X)=0
\end{align}
\item By the modified characterization of the filtrations we have the inclusions 
\begin{equation}\label{inclusion}
\widetilde{N}^{j}H^i(X)\subseteq N^jH^i(X)\subseteq F^jH^i(X, \mathbf{C})\cap H^i(X).
\end{equation}
We say that $X$ satisfies \textbf{Generalized Hodge Conjecture} proposed by A. Grothendieck (denoted by \textbf{GHC($X, p, i$)}), if
the maximal $\mathbf{Q}$-Hodge structure contained in $F^pH^i(X, \mathbf{C})\cap H^i(X)$ is $N^pH^i(X)$.

 By the Lefschetz theorem the cup product by the cohomology class by an ample line bundle on $W$ gives an isomorphism of $H_i(W)\to H^i(W) $ for $i=0,1$. Hence $\widetilde{N}^{\lfloor i/2\rfloor} H_i(X)=N^{\lfloor i/2\rfloor} H_i(X).$ More generally, it was conjectured by Friedlander and Mazur~\cite[p. 71]{frma} that these two filtrations agree for $k=\mathbf{C}$. Vial~\cite[Proposition 1.1]{v} has shown that, if the Lefschetz standard conjecture holds for all smooth projective varieties of dimension $<$ dim $X$, then the two filtrations agree if and only if the Lefschetz standard conjecture holds for $X$. 
\end{enumerate}
\subsection{Chow-K\"unneth Decomposition}\label{CK}
Next we recall some terminology from the theory of Chow motives.
\begin{definition}
$X$ of dimension $d$ is said to possess a Chow-K\"unneth decomposition, if there is a collection $\{{\pi_i}\}_{i=0}^{2d}$ of codimension $d$ cycles on $X\times X$ such that the following relations hold in $CH^d(X\times X).$
\begin{enumerate}
\item $\displaystyle\Delta_{X}=\sum_{i=0}^{2d} \pi_i$, where $\Delta_X$ is the class of the diagonal of $X$,
\item ${^t{\pi_i}}=\pi_{2d-i}$, (duality)
\item $\pi_i\circ\pi_i=\pi_i$ (idempotent), $\pi_i\circ\pi_j= 0$ if $i\neq j$ (mutually orthogonal) and 
\item $\pi_i$ acts on $H^j(X, \mathbf{Q})$ as $\delta_{ij}Id_{H^j(X, \mathbf{Q})}$. 
\end{enumerate}
\end{definition}
Some examples of smooth projective varieties admitting a Chow-K\"unneth decomposition are curves, surfaces. 
For abelian varieties $X$, we have the multiplication map by $n:X\to X$ for $n\in\mathbf{Z}^{\ast}$, which induces $n_{\ast}: CH^i(X)\to CH^i(X).$ 
A. Beauville proved the following theorem:
\begin{theorem}[Beauville~\cite{b}]\label{bdec}
For $X$ an abelian variety of dimension $d$, there is a decomposition of $CH^i(X)$ given by  
\begin{center}
$\displaystyle CH^i(X)=\bigoplus_{s=i-d}^iCH^i(X)_s$,
\end{center}
where 
\begin{center}
$CH^i(X)_s=\{\alpha\in CH^i(X)~|~n_{\ast}\alpha=n^{2d-2i+s}\alpha$ for all $n\in\mathbf{Z}^{\ast}\}.$
\end{center}
\end{theorem}
\begin{conjecture}[Beauville's Conjecture]\label{beauconj}
$CH^i(X)_s=0$ for $s<0.$
\end{conjecture}

We further have by Shermenev~\cite{sh}, Denninger-Murre~\cite{dm},
\begin{theorem}\label{CK-Ab}
For $X$ an abelian variety of dimension $d$ over $k=\bar{k}$,
\begin{enumerate}
\item X has a Chow-K\"unneth decomposition
\item There exist Chow-K\"unneth components $\pi_i$ such that $n_{\ast}\circ\pi_i=\pi_i\circ n_{\ast}=n^{i}\pi_i$ for all $n$ and such $\pi_i$'s are unique.
\end{enumerate}
\end{theorem}
Combining this with Beauville's result, we get that 
\begin{equation}\label{dec}
CH^i(X)_s=(\pi_{2i-s})_{\ast}CH^i(X).
\end{equation}
\begin{corollary}
\begin{enumerate}
\item 
$\pi_i$ operates as 0 on $CH^j(X)$ for $i>j+d$ and $i<j.$
\item 
For $2j< i<j+d$, $\pi_i$ operates as 0 on $CH^j(X)$ iff the Beauville's conjecture~\ref{beauconj} holds true for $CH^j(X)$.
\end{enumerate}

\end{corollary}
\begin{proof}
\begin{enumerate}
\item $\pi_i$ acts different from 0 on $CH^j(X)$ only if, there is $s$ such that $i=2j-s$. But $i>d+j$ iff $2j-s>d+j$ iff $s<j-d$. But by Beauville's theorem $s\geq j-d$. Similarly, $i<j$ iff $2j-s<j$ iff $s> j$, which again does not occur in Beauville's decomposition. 
\item Similar argument as before: we have $s\geq0$ iff $2j-s\leq2j$. Thus by the assumption on the range of $i$, $i\neq 2j-s$, hence $\pi_i$ operates as 0 on $CH^j(X)$ iff $CH^j(X)_s=0$ for $s<0$. 
\end{enumerate}
\end{proof}
\begin{remark}\label{CKoperate}
It follows from combining the above results that, $\pi_i$  operates as 0 on $CH_0(X)$ for $i< d$. Further $\pi_i$ operates as 0 on $G^d$ for $i\neq d$ and $\pi_d$ operates as the identity on $G^d.$
\end{remark}
\end{section}

\begin{section}{Refined Chow-K\"unneth Decomposition} \label{refinedCK}
In this section, we will consider $X$ to be a smooth projective variety over an algebraically closed field $k$ and will concentrate on $i=d$, $H^d(X)$ the middle singular cohomology of $X$. Given a Chow-K\"unneth decomposition of $X$, under some further assumptions ($\ast$) and ($\ast\ast$) defined below, we follow C. Vial's recipe in~\cite{v} to produce a refined decomposition of the cycle $\pi_d$ into mutually orthogonal, idempotents given by 
\begin{center}
$\pi_d=\displaystyle\sum_{j=0}^{\lfloor d/2\rfloor}\Pi_{d,j}$
\end{center}
such that $(\Pi_{d,j})_{\ast}H^{\ast}(X,\mathbf{Q})=Gr^j_{\widetilde{N}^{\bullet}}H^d(X).$

Let us setup some assumptions. 

Let $L$ be an ample line bundle on $X$ which defines an embedding $X\subseteq \mathbf{P}_k^N$ and hence for any integer $i\leq d$, a map $L^{d-i}: H^i(X)\to H^{2d-i}(X)$ given by intersecting $d-i$ times with the cohomology class of a hyperplane section $H$ of $X$ . This map is induced by a correspondence on $X\times X$ given by $X\times H$ and is an isomorphism of Hodge structures (Hard Lefschetz theorem). We say $X$ satisfies the property $B_i$ if the isomorphism $L^{i-d}:=(L^{d-i})^{-1}: H^{2d-i}(X)\to H^{i}(X)$ is induced by an algebraic correspondence. If $X$ satisfies property $B_i$ for every $i$, then we say $X$ satisfies property \textbf{B}. 

The property \textbf{B} is satisfied by curves, surfaces (Grothendieck, see~\cite{KL68}), abelian varieties (Lieberman~\cite{lib}, Kleiman~\cite{KL68}), complete intersections and any products, hypersurface intersections or their finite quotient. One of Grothendieck's conjectures (also known as the \textbf{Lefschetz standard conjecture}) says that all smooth projective varieties satisfy the property \textbf{B}.

Assumption ($\ast$):
\begin{enumerate}
\item $X$ satisfies \textbf{B} and
\item for 
all $j\geq 1$,  
       either there exists $W_{d,j}$ as in the Equation~(\ref{nivsimple}) satisfying $B_l$ for all $l\leq d-2j-2$, or
     $N^{j+1}H^d(X)=\widetilde{N}^{j+1}H^d(X).$
         \end{enumerate}

As the property $B_1$ holds for all smooth projective varieties, the property ($\ast$) holds for all smooth projective varieties of dimension at most 5 that satisfy property \textbf{B}, as the $W_{d,j}$ need to only satisfy $B_l$ for $l=1$. In particular property ($\ast$) holds for curves, surfaces, abelian varieties of dimension $\leq 5$, complete intersections of dimension $\leq 5$, uniruled 3-folds, rationally connected 5-folds with $H^3(X, \Omega_X)=0$. We will list further examples of abelian varieties for which the assumption ($\ast$) holds true in the Section~\ref{list}.

Further, we consider varieties $X$ satisfying the 
\begin{center}
Assumption ($\ast\ast$):
$Ker (cl: CH^d(X\times X)\to H^{2d}(X\times X)) $ is a nilpotent ideal.
\end{center}
Here $CH^d(X\times X) $ is considered as a ring with respect to the multiplication given by composition of correspondences, denoted by $\circ$. For this ring structure, the above kernal subgroup associated to the cycle class map $cl$ is an ideal, as stated above the assumption ($\ast\ast$) states that this ideal is nilpotent. For example, varieties dominated by a product of curves satisfy ($\ast\ast$). 
More generally, any variety satisfying Kimura's finite dimensionality as in~\cite{k} satisfies ($\ast\ast$). Conjecturally, this holds for arbitrary smooth projective varieties.

In the paper, we will be using that abelian varieties satisfy the assumption ($\ast\ast$), as shown by Kimura in~\cite{k}, that is why we will not explicitly mention this particular condition in our applications. 
Let us recall a lifting lemma: 
\begin{lemma}{~\cite[Lemma 5.4]{j}}
\label{lem}
Let $X$ be a smooth projective variety satisfying the property ($\ast\ast$) as stated above. Let $c_1,\dots, c_n\in CH^d(X\times X)$ be correspondences such that $cl(c_i)\in H^{\ast}(X\times X)$ define mutually orthogonal idempotents adding to the identity. Then there exists mutually orthogonal idempotents $p_1,\dots, p_n\in CH^d(X\times X)$ adding to the identity such that $cl(p_i)=cl(c_i)$ for all i. Moreover any two such choices $\{p_1,\dots,p_n\}$ and $\{p_1',\dots,p_n'\}$ are conjugate by an element lying above the identity i.e. there exists a nilpotent correspondence $\eta\in CH^d(X\times X)$ such that $p'_i=(1+\eta)\circ p_i\circ(1+\eta)^{-1}$ for all i.
\end{lemma}
We next recall Vial's theorem, 
\begin{theorem}{~\cite[Theorem 1.2]{v}} \label{vial}
Let $X$ be a smooth projective variety for which the properties ($\ast$) and ($\ast\ast$) hold as defined above. Then there exists a collection of codimension $d$ cycles $\{\Pi_{d,j}\}_{j=0}^{\lfloor d/2\rfloor}$ on $X\times X$ which are mutually orthogonal, idempotent, 
such that $\pi_d=\sum_j \Pi_{d,j}$ in $CH^d(X\times X)$ and ${\Pi_{d,j}}_{\ast}H^{\ast}(X) =Gr^j_{\widetilde{N}^{\bullet}}H^d(X)$. For any such choice of idempotents we have:
\begin{enumerate}
\item $\Pi_{d,j}$ acts as 0 on $CH_l(X)$ if either $l<j$ or $l>d-j$.
\item $\Pi_{d,j}$ acts as 0 on $CH_l(X)$ if $l=d-j$ and $d<2l$.
\item $\Pi_{d,j}$ acts as 0 on $CH_l(X)$ if $l+1=d-j$ and $d\leq 2l$.
\end{enumerate}
 \end{theorem}
\subsection{A List of Complex Abelian Varieties and the Assumption ($\ast$)}
 \label{list}
 We will consider the following list of complex abelian varieties to which we can apply the Theorem~\ref{vial} i.e. for which the assumption ($\ast$) holds true (see Remark~\ref{ast}). We recall again that the property \textbf{B} holds for all abelian varieties~\cite{lib}.

 \paragraph{A List:}
 \begin{enumerate}
 \item An abelian variety $X$ of dimension $\leq 5$.
 \item An abelian variety $X$ for which the Hodge group (denoted by $Hg(X)$) is equal to the symplectic group on the vector space $H_1(X,\mathbf{Q})$ with polarization $\beta$, denoted by $Sp(H_1(X,\mathbf{Q}), \beta)$. 
 \\
 Next 
 we list abelian varieties $X$, for which the Generalized Hodge Conjecture is known for all powers of $X$ (see~\cite{gord} and the references therein).\\
 For a complex abelian variety $X$, let $D(X)$= the endomorphism algebra $End(X)\otimes_{\mathbf{Z}}\mathbf{Q}$. We say 
 \begin{enumerate}
 \item $X$ is of type I if $D(X)$ is a totally real field and
 \item $X$ is of type II if $D(X)$ is a totally indefinite quaternion algebra over a totally real number field.
 \end{enumerate}
  \item An abelian variety $X$ such that the Hodge ring of $X^n$ is generated by divisors on $X^n$ for all $n\geq1$ (known as stably non-degenerate~\cite[Theorem 2.7]{haz84}), and all of whose simple components are of type I or type II.   
\item An abelian variety $X$ such that dim $X=$ dim $Hg(X)$ (called non-degenerate as in Ribet~\cite[Corollary 3.6, p. 87]{ribet}), of CM type with CM field $E=D(X)$ so that, $E$ is an imaginary quadratic field over a totally real field $F$ of degree $d$ over $\mathbf{Q}$, such that the degree $[\bar{E}:\bar{F}]=2^d$, where bars denote the Galois closure of respective fields. (~\cite[Example 1]{abdul05}).
 \item  A product of complex elliptic curves.
 \end{enumerate}

\begin{theorem}\label{234}
Let $X$ be an abelian variety in the cases 2, 3, 4, 5 defined above. Then for all $j$ and $i$,
\begin{center}
$\widetilde{N}^{j}H^i(X)={N}^{j}H^i(X)$. 
\end{center}
\end{theorem}
\begin{proof}
\begin{itemize}
\item For $X$ in Case 2:
The proof follows from~\cite[Proposition 4.4]{fr} and~\cite[Theorem 7.3]{frma} after one observes that in loc. cit. 
$\widetilde{N}^{p}H_r(X)=C_pH_r(X)$ and ${N}^{p}H_r(X)=G_{p}H_r(X)$.  
\item For $X$ in Case 3: For the proof, we refer to~\cite[Theorem 8.1]{abdul00}.
\item For $X$ is in Case 4: The proof follows from~\cite[Example 1]{abdul05}, where it is proved that such $X$ is dominated by powers of $X$ and since $X$ is non-degenerate (Ribet~\cite[Corollary 3.6, p. 87]{ribet}), GHC holds for all powers of $X$. Now the conclusion follows from~\cite[Proposition 5.1 and Theorem 5.2]{abdul00}. 
\item For $X$ is in Case 5: The proof follows from~\cite[p. 195]{lawsonJr}.
\end{itemize}
\end{proof}
\begin{remark}
\begin{enumerate}
\item As mentioned earlier, $X$ in case 1 satisfies the assumption ($\ast$).
\item If $X$ belongs to the cases 2, 3, 4, 5, then $X$ satisfies the assumption ($\ast$), as it follows from the above Theorem~\ref{234} for $i=d=$ dim $X$. 
\end{enumerate}
\label{ast}
\end{remark}


Applying the Theorem~\ref{vial} to the abelian varieties satisfying the assumption ($\ast$), in particular to those in the List above, we get:
\begin{corollary}\label{vc}
 Let $X$ be a complex abelian variety satisfying the assumption ($\ast$). Then there exists a collection of cycles $\{\Pi_{d,j}\}_{j=0}^{\lfloor d/2\rfloor}$ on $X\times X$ of codimension $d$ which are mutually orthogonal, idempotent, such that $\displaystyle\pi_d=\sum_{j=0}^{\lfloor d/2\rfloor}\Pi_{d,j}$ in $CH^d(X\times X)$ and $(\Pi_{d,j})_{\ast}H^{*}(X)=Gr^j_{\widetilde{N}^{\bullet}}H^d(X).$ Moreover, the idempotents satisfy:
\begin{enumerate}
\item $\Pi_{d,j}$ acts as 0 on $CH_l(X)$ if either $l<j$ or $l>d-j$.
\item $\Pi_{d,j}$ acts as 0 on $CH_l(X)$ if $l=d-j$ and $d<2l$.
\item $\Pi_{d,j}$ acts as 0 on $CH_l(X)$ if $l+1=d-j$ and $d\leq 2l$.
\end{enumerate}
  \end{corollary}
   \begin{remark}
 Combining the above corollary with the Remark~\ref{CKoperate}, it follows that:
 \begin{enumerate}
 \item $\pi_i$ operates as 0 on $CH_0(X)$ for $i<d$ and $\pi_i$ operates as 0 on $G^d$ for $i\neq d$.
 \item $\pi_d$ operates nontrivially on $CH_0(X)$, but operates as identity on the subgroup $G^d\subseteq CH_0(X).$
 \item $\Pi_{d, j}$ operates as 0 on $CH_0(X)$ for $j>0$.
 \item $\Pi_{d, 0}$ operates as $\pi_d$ on $CH_0(X)$. 
 \end{enumerate}
\label{remCKoperate}
  \end{remark}
   

We will further need a notion of symmetrically distinguished cycles on abelian varieties for our Theorem~\ref{main}.
\end{section}

\begin{section}{Cycles on Abelian Varieties}
\subsection{Symmetrically Distinguished Cycles on Abelian Varieties}
Throughout this section, $X$ will denote an abelian variety over $k.$ The notion of symmetrically distinguished cycles was first introduced by P. O'Sullivan in~\cite[Definition 6.2.1]{o}. 
\begin{definition} 
\label{sd}
Let $\alpha$ be an element in the $i^{th}$ Chow group $CH^i(X)$ of $X$. Let $V_m(\alpha)$ be the $\mathbf{Q}$-subspace of $CH(X^m)$ generated by elements of the form 
 \begin{center}
  $p_{*}(\alpha^{r_1}\otimes\alpha^{r_2}\otimes\dots\otimes\alpha^{r_n})$
 \end{center}
where $n\leq m$, the $r_j$ are the integers $\geq 0$, and $p: X^n\to X^m$ is a closed immersion with each component $p:X^n\to X^m\to X$ either a projection or the composite of a projection
with the inversion morphism $(-1)_X:X\to X$. Then $\alpha$ will be called the \textbf{symmetrically distinguished} element if for every $m$ the restriction of the quotient $CH(X^m)\to CH(X^m)_{num}$
restricted to $V_m(\alpha)$ is injective. An arbitrary element of $CH^{*}(X)$ will be called \textbf{symmetrically distinguished} if each of its homogeneous component is symmetrically distinguished. 
\end{definition}

The main result of O'Sullivan's paper \cite[Corollary 6.2.6]{o} is the following theorem:
\begin{theorem}

\begin{enumerate}
 \item  For every cycle in $CH_{num}(X)$ there exists a unique symmetrically distinguished cycle in $CH(X)$ which lies above it.
 \item The symmetrically distinguished cycles in $CH(X)$ form a $\mathbf{Q}$-vector subspace, and the product of symmetrically distinguished cycles in $CH^{i}(X)$ and $CH^{j}(X)$ is a symmetrically distinguished cycle in $CH^{i+j}(X)$.
 \item For any homomorphism of abelian varieties $f:X\to Y$, the pullback $f^{*}$ and push forward $f_{*}$ preserve the symmetrically distinguished cycles.
\end{enumerate}
\label{symdist}
  \end{theorem}
 \begin{remark} 

 \begin{enumerate}
 \item
  From this theorem, we can deduce that if $f:X\to X$ is an automorphism of abelian variety $X$, then the graph of $f$ denoted by $\Gamma_f$, as an element in $CH^d(X\times X)$ is a symmetrically distinguished. Indeed the class $[X] \in CH^0(X)$ is a symmetrically distinguished element and $[\Gamma_f]= (1_X\times f)_{*}([X])$. Thus the claim follows from Theorem~\ref{symdist} (3). We make this remark here, separately for future reference. Note that in the above remark, it is enough to assume $f$ is any homomorphism of abelian varieties, not necessarily an automorphism.
\item Another important example of a symmetrically distinguished cycle on $X$ is a symmetric divisor $D$ on $X$ i.e. $D\in CH^1(X)$ such that $(-1)_X^{\ast}D=D$ in $CH^1(X)$. Hence, cycles on $X$ generated by symmetric divisors on $X$ are symmetrically distinguished as it follows from Theorem~\ref{symdist} (2).
  \end{enumerate}
 \label{remsd}
 \end{remark}
 \subsection{Main Theorem for Abelian Varieties}
The \textbf{main theorem} of this paper is the following:
\begin{theorem}
\label{main}
 Let $X$ be a complex abelian variety satisfying the assumption ($\ast$).
Suppose $\Gamma\in CH^d(X\times X)$ is a correspondence from $X$ to $X$ which is symmetrically distinguished. Let 
\begin{center}
$[\Gamma]_{\ast}: \dfrac{H^d(X)}{\widetilde{N}^1H^d(X)} \to \dfrac{H^d(X)}{\widetilde{N}^1H^d(X)}$
\end{center}

be the induced morphism which is 0. 
Then $\displaystyle[\Gamma]_{*}=0: \cap_{i=1}^{d} Pic^0(X) \to  \cap_{i=1}^{d} Pic^0(X)$ induced by the restriction of $[\Gamma]_{*}:CH_0(X)\to CH_0(X)$.
\end{theorem}

The proof of Theorem~\ref{main} will follow from the following proposition.
\begin{proposition} 
\label{sdv}
 Let $X$ be a complex abelian variety satisfying the assumption ($\ast$). Then Theorem~\ref{vial} holds for $X$ and further there exist $\Phi_{d,j}$ symmetrically distinguished in the sense of Definition~\ref{sd} such that $\Phi_{d,j}\in CH^d(X\times X)$ are mutually orthogonal, idempotents and $\displaystyle\pi_d =\sum_{j} \Phi_{d,j}$ in $CH^d(X\times X)$ and the properties 1 to 3 of Corollary~\ref{vc} hold for these modified $\Phi_{d,j}$'s when replaced by $\Pi_{d,j}.$
\end{proposition}
\begin{proof}
By applying Theorem~\ref{symdist} (1) for $X\times X$, can lift the elements $\Pi_{d,j}$ as in Theorem~\ref{vial} modulo homological equivalence (recall that homological equivalence is equivalent to numerical equivalence for abelian varieties~\cite{lib}) to $CH^d(X\times X)$ such that $\exists$ unique $\Phi_{d,j}\in CH^d(X\times X)$ which are symmetrically distinguished and $\Phi_{d,j} = \Pi_{d,j}$ in $CH^d(X\times X)_{num}$.  

Now let us check that the $\{\Phi_{d,j}\}$ satisfy all the properties of Theorem~\ref{vial}. To see that $\Phi_{d,j}\circ\Phi_{d,s}=\delta_{js}$, for Kronecker symbol, note that this equation holds modulo homological equivalence as it holds for $\Pi_{d,j}$ replaced by $\Phi_{d,j}.$ Since both the cycles in the first equation are symmetrically distinguished, these are equal in $CH^d(X\times X)$. 

Similarly, one can see that $\Phi_{d,j}$'s are idempotents. By Lemma~\ref{lem} $\{\Phi_{d,j}\}_{j}$ and $\{\Pi_{d,j}\}_{j}$ are conjugate  and satisfy the properties in Corollary~\ref{vc} by \cite[1.4.4]{v}.
\end{proof}
Now for the abelian variety $X$, we have a Chow-K\"unneth decomposition $\{\pi_i\}_{i=0}^{2d}$ such that $\Delta_X= \sum_i\pi_i$ by Theorem~\ref{CK-Ab}. By above discussion $\pi_d=\sum_j\Pi_{d,j}$ in $CH^d(X\times X)$,  where $\pi_i, \Pi_{d,j}$'s are symmetrically distinguished as in Proposition~\ref{sdv}. \\
Now we are all set to prove our main theorem.
\paragraph{Proof of the main theorem~\ref{main}:}
\begin{proof}
Consider the following expressions in the ring $CH^d(X\times X)$ under composition of correspondences denoted by $\circ.$
\begin{align}
\Gamma\circ\Pi_{d,0}&= \Gamma\circ(\pi_d-\sum_{r>0}\Pi_{d,r})\\
                                              &=\Gamma\circ\pi_d-\sum_{r>0}\Gamma\circ\Pi_{d,r}
\end{align}

Each of the terms in the above expression induces an endomorphism of $CH_0(X)$, and by the Remark~\ref{remCKoperate}, it follows that as endomorphisms of $CH_0(X)$: 
\begin{center}
$[\Gamma\circ\Pi_{d,0}]_{\ast}=[\Gamma\circ\pi_d]_{\ast}:CH_0(X)\to CH_0(X)$.

\end{center}

Next we show that $\Gamma\circ\Pi_{d,0}$ is 0 in $CH^d(X\times X)$.

Consider the cycle $\Gamma\circ\Pi_{d,0}\in CH^d(X\times X)$ modulo numerical equivalence (equivalently homological equivalence). By the assumption on $\Gamma$, $[\Gamma]_{\ast}$ is 0 on the image of $[{\Pi_{d,0}}]_{\ast}$, which is $\dfrac{H^d(X)}{\widetilde{N}^1H^d(X)}$, by the Corollary~\ref{vc} for $j=0$. Hence $\Gamma\circ\Pi_{d,0}$ induces the 0 morphism $H^{\ast}(X)\to H^{\ast}(X)$. 

Recall from the Proposition~\ref{sdv}, and the remark following it, that the cycle $\Gamma\circ\Pi_{d,0}$ is symmetrically distinguished. Hence it is numerically equivalent to 0, which implies it is actually rationally equivalent to 0.  

Thus, the endomorphism of $CH_0(X)$ induced by such a cycle is again 0. Hence by earlier discussion, 
$[\Gamma\circ\pi_d]_{\ast}:CH_0(X)\to CH_0(X)$ is 0. In other words, $[\Gamma]_{\ast}:CH_0(X)\to CH_0(X)$ is 0 on the image of $[\pi_d]_{\ast}:CH_0(X)\to CH_0(X).$ We know that Im $([\pi_d]_{\ast}:CH_0(X)\to CH_0(X))=\cap_{i=1}^{d} Pic^0(X)$ i.e. 0-cycles which are rationally equivalent to sums of intersection of $d$ divisors of degree 0, where the equality follows from~\cite[Proposition 4(a)]{b}.
Hence $[\Gamma]_{*}: \cap_{i=1}^{d} Pic^0(X) \to  \cap_{i=1}^{d} Pic^0(X)$ is 0 map.
\end{proof}

In the remaining part of this section, we apply the Theorem~\ref{main} to particular symmetrically distinguished cycles. 
\begin{corollary}
\label{mc}
Let $f :X \to X$ be an automorphism of $X$ where $X$ is a complex abelian variety of dimension $d$ satisfying the assumption ($\ast$) as in the Section~\ref{refinedCK}.
Suppose the induced morphism 
\begin{center}
$f_{\ast}: \dfrac{H^d(X)}{\widetilde{N}^1H^d(X)} \to \dfrac{H^d(X)}{\widetilde{N}^1H^d(X)}$
\end{center}
 is the identity. Then $f_{*}=Id: \cap_{i=1}^{d} Pic^0(X) \to  \cap_{i=1}^{d} Pic^0(X)$ induced by the restriction of $f_{*}: CH_0(X)\to CH_0(X)$. 
\end{corollary}
\begin{proof}
Take $\Gamma=\Gamma_f-\Delta_X$ in the Theorem~\ref{main}. Now $\Gamma$ is symmetrically distinguished, which follows from the Remark~\ref{remsd}(1). Hence, the corollary follows from Theorem~\ref{main}.
\end{proof}

\begin{remark}
\begin{enumerate}
\item The Corollary~\ref{mc} is non-trivial, only in the Cases 1, 3, 4, 5 in the List in Section~\ref{list}, because in the Case 2, there are no non-trivial automorphisms of the abelian varieties.
\item In view of the Remark~\ref{remsd}, it is an interesting question to write down other examples of symmetrically distinguished cycles which satisfy the assumptions of the Theorem~\ref{main}.    
\end{enumerate}
\end{remark}

\subsection{Examples}\label{eg}
In this section, we will see some examples where the assumptions of the Corollary~\ref{mc} hold true.  
\begin{example}
\item 
Let $E$ be an elliptic curve over $\mathbf{C}$ given by $\mathbf{C}/\Lambda$ for a lattice $\Lambda$ in $\mathbf{C}$. Let $X=E^d=\mathbf{C}^d/\Lambda^d$. Let $A\in SL_d(\mathbf{Z})$ be a non-trivial element of the special linear group $SL _d(\mathbf{Z})$. Now $A$ acts naturally on $\mathbf{C}^d$ such that $A(\Lambda^d)\subset\Lambda^d$.  Thus $A$ induces $f\in Aut(X)$. Let $\gamma_1,\cdots, \gamma_d$ in $\mathbf{C}$ be the eigenvalues of $A$. Given co-ordinates $z_1,\cdots, z_d$ on $\mathbf{C}^d$, we get $dz_1,\cdots, dz_d$ as $\mathbf{C}$-basis for $H^0(X, \Omega_X)$. Now $f$ induces action on the $\mathbf{C}$-vector space $H^0(X, \Omega_X)$ with eigenvalues $\gamma_1,\cdots, \gamma_d$. Hence, $f_{\ast}$ acts on $ H^0(X, \Omega_X^d)\simeq\displaystyle\wedge_1^d H^0(X, \Omega_X)$ by det $A=\gamma_1\cdots\gamma_d=1$. 
Hence, we have $f_{\ast}=Id: H^0(X, \Omega_X^d)\to H^0(X, \Omega_X^d)$.
\end{example}
\begin{example}\label{foreg}
Next, we will give other examples of situations where $f$ is an automorphism of $X$ and $f_{\ast}:H^d(X)\to H^d(X)$ is not the identity but $f_{\ast}: \dfrac{H^d(X)}{\widetilde{N}^1H^d(X)}\to \dfrac{H^d(X)}{\widetilde{N}^1H^d(X)}$ is Id. 
Let us recall a construction of abelian varieties with real multiplication. We follow the exposition from~\cite[1.13.5]{gord}.\\
\textbf{Abelian varieties with real multiplication:}
Let $K$ be a totally real number field such that deg~$[K:\mathbf{Q}] =d$. Let $\alpha\mapsto \alpha_{(j)}:K\to\mathbf{R}$ be the distinct $\mathbf{R}$-embeddings for $1\leq j\leq d$. Let $\tau_j\in\mathbf{C}$ be fixed such that $Im(\tau_j)>0$ for $1\leq j\leq d$. For the ring of integers $\mathcal{O}$ of $K$,
define 
\begin{center}
$\phi: \mathcal{O}\oplus \mathcal{O}\to \mathbf{C}^d$ by $(\alpha,\beta)\mapsto (\alpha_{(1)}\tau_1+\beta_{(1)},\cdots,\alpha_{(d)}\tau_d+\beta_{(d)}).$
\end{center}
 Define $L:=Im~\phi$. Note that $L$ is a lattice in $\mathbf{C}^d$. Define $X:=\mathbf{C}^d/L$.
One can define a Riemann form on $X$ which will imply that the complex torus $X$ is a complex abelian variety. 

Now the inclusion $K\to End^0(X)$ restricts to $\mathcal{O}\to End(X)$. So that for the units of $\mathcal{O}, \mathcal{O}^{\ast} \to Aut(X)$. By Dirichlet's unit theorem, have that $\mathcal{O}^{\ast}$ has rank $d-1$. So if $d>1$, one can find $\gamma \in \mathcal{O}^{\ast}$ which is of infinite order. 

The element $\gamma\in \mathcal{O}^{\ast}$ defines $f_{\gamma}\in Aut(X)$ given by 
\begin{center}
$f_{\gamma}(z_1,\cdots,z_d)=(\gamma_{(1)}z_1,\cdots,\gamma_{(d)}z_d)$. 
\end{center}

Now we will consider the action of $f_{\gamma}$ on the cohomology of $X$. $f_{\gamma}$ acts on $X$ via a matrix $M_{\gamma}\in GL_d(\mathbf{C})$ with eigenvalues $\{\gamma_{(1)},\cdots,\gamma_{(d)}\}.$  So the eigenvalues of $f_{\gamma}$ on $H^d(X)$ are given by various products of these eigenvalues which in general need not be 1, which implies that in general $f_{\gamma}$ is not identity on $H^d(X)$. The induced action of $f_{\gamma}$ on $H^{0}(X,\Omega_X^d)$ is given by a scalar with the eigenvalue $\gamma_{(1)}\gamma_{(2)}\cdots\gamma_{(d)}=$ det $M_{\gamma}$ which is also the Norm of $\gamma$, $Nm_{K/\mathbf{Q}}(\gamma)$. Since $\gamma\in\mathcal{O}^{\ast}$, $Nm_{K/\mathbf{Q}}(\gamma)=\pm1$. So if $Nm_{K/\mathbf{Q}}(\gamma)=1$, take $\gamma$ and when $Nm_{K/\mathbf{Q}}(\gamma)=-1$, can take $\gamma^2$ in $\mathcal{O}^{\ast}$ so that $Nm_{K/\mathbf{Q}}(\gamma^2)=1$. Thus the respective eigenvalue on $H^{0}(X,\Omega_X^d)$ is 1.   

More concretely, let $K=\mathbf{Q}(\sqrt 7)$ and so its ring of integers is $\mathcal{O}=\mathbf{Z}[\sqrt7]$. Here $d=2=dim [K:\mathbf{Q}]. $
Take the fundamental unit $\gamma= 8+3\sqrt 7$ of $\mathcal{O}.$ So $\gamma_{(1)}=8+3\sqrt 7$ and $\gamma_{(2)}=8-3\sqrt 7$. So the determinant of the matrix $M_{\gamma}=1$ (We have taken $\gamma$ such that the norm $Nm_{K/\mathbf{Q}}(\gamma)=1$). But eigenvalues of $M_{\gamma}$ on $H^2(X)$ are $\{{\gamma_{(1)}}^2,{\gamma_{(2)}}^2, 1\}$. One can check ${\gamma_{(1)}}^2>1.$

So we are in a situation where the automorphism $f_{\gamma}$ of a complex abelian surface $X$ acts non-trivially on $H^2(X)$, but acts trivially on $H^0(X, \Omega_X^2)$.
\end{example}
In the next section, we discuss some applications. 
 \end{section}
 
 \begin{section}{Applications}\label{App}

\subsection{Cycles using Pontryagin Product}
For an abelian variety $X$ of dimension $d$ over $k$, an algebraically closed subfield of $\mathbf{C}$, S. Bloch~\cite[Lemma~1.2(c)]{bl1} proved that
 $G^d:=Pic^0(X)^{\cap d}$ is generated by expressions of the form 
\begin{center}
$((a_1)-(0))\ast((a_2)-(0))\ast\cdots\ast((a_d)-(0))$ 
\end{center}
where $a_1,\cdots, a_d\in X(k)$ and $\gamma\ast\gamma'$ is the Pontryagin product of cycles $\gamma$ and $\gamma'$.  
Thus, for $f$ as in Corollary~\ref{mc}, we get identities
\begin{align}
((a_1)-(0))\ast((a_2)-(0))\ast\cdots\ast((a_d)-(0))=& f_{\ast}(((a_1)-(0))\ast((a_2)-(0))\ast\cdots\ast((a_d)-(0)))\\
                                                                            =& ((f(a_1))-(0))\ast((f(a_2))-(0))\ast\cdots\ast((f(a_d))-(0)).
\end{align}
In particular, for $d=2$, we have a corollary:
\begin{corollary}
Let $f$ be an automorphism of a complex abelian surface $X$ such that $f_{\ast}$ acts trivially on $H^0(X, \Omega_X^2)$, we have identities between rational equivalence classes of 0-cycles
\begin{center}
$(f(a+b)) - (f(a)) - (f(b)) = (a+b) - (a) - (b)$. 
\end{center}
\end{corollary}
The main point of this particular case is that, a priori the two cycles on LHS and RHS are different, but still rationally equivalent to each other in $CH_0(X)$. This does not seem to be clear by a direct argument even if $X=E\times E$ for an elliptic curve $E$, and $f\in SL_2({\mathbb Z})\subseteq Aut(X)$ is an arbitrary element, which in general yields a non-trivial automorphism of $CH_0(X)$.
  \subsection{Action of Automorphisms of Kummer surfaces or generalized Kummer varieties on $CH_0$}
   
  Here we consider complex projective Kummer surfaces and their higher dimensional analogs called \textit{generalized Kummer varieties}. 
  Huybrechts~\cite{h}-Voisin~\cite{vo1} have proved the following theorem:
\begin{theorem}
\label{hv}
Let $f$ be an automorphism of projective K3 surface $X$ of finite order which is symplectic (i.e. which acts as the identity on $H^0(X, \Omega_X^2)$, then $f_{\ast}$ acts as the identity on $CH_0(X)_0$.
\end{theorem}
We would like to prove a similar theorem for Kummer surface $Km(X)$ associated to a complex abelian surface $X$, but for (possibly infinite order) automorphisms induced from the complex abelian surface $X$.

  First let us set up some notation. For a complex smooth projective variety $Y$ of dimension $d$, one can associate a complex abelian variety called Albanese variety $Alb(Y)$ which admits a morphism from $alb_Y: CH_0(Y)_0\to Alb(Y)$, where $CH_0(Y)_0 :=$ ker $(deg:CH_0(Y)\to H^{2d}(Y))$. 
  Let $T(Y)$ denote the kernel of the Albanese map $alb_Y$. 
  \paragraph{Kummer surface} 
  Let us recall the definition of $Km(X)$: Let $\imath: X\to X$ be the inversion map on $X$. So the finite group $G:=\{id, \imath\}$ naturally acts on $X$ with fixed point set given by 2-torsion points of $X$. Thus the quotient scheme $X/G$ has singularity set given by the 16 2-torsion points. Now we can blow up $X/G$ along the closed subscheme formed by these 16 points, which we call the Kummer surface $Km(X)$ associated to $X$. Now let $f$ be an automorphism of the abelian surface $X$. Thus $f$ commutes with $\imath$. Hence $f$ fixes the fixed point set of $G$. So $f$ lifts to an automorphism $\tilde{f}$ of $X/G$ which fixes the closed subscheme given by the fixed points of $G$. Hence one can lift the automorphism $\tilde{f}$ to that of the Blow up $Km(X)$ given by $g$. So we get the commutative diagram
       \begin{equation}
\label{4}
\begin{tikzpicture}
[back line/.style={densely dotted},cross line/.style={loosely dotted}]
\matrix (m) [matrix of math nodes, row sep=2 em,column sep=2.em, text height=1.5ex, text depth=0.50ex]
 {
  Km(X) & Km(X) \\
 X/G & X/G \\
 };
\path[->,font=\scriptsize]
  (m-1-2) edge node[auto] {} (m-2-2)
  (m-1-1) edge node [auto] {} (m-2-1)
  (m-1-1) edge node [auto] {$g$} (m-1-2)

(m-2-1) edge node [auto] {$\tilde{f}$} (m-2-2);
\end{tikzpicture}
\end{equation}
  
Now we have a theorem by Bloch~\cite[Corollary A.10]{blochetal} which says that the rational map $\phi: X\dashrightarrow Km(X)$ induces a surjective isogeny $\phi^{\ast}: T(Km(X))\to T(X).$ More precisely $\phi^{\ast}\circ \phi_{\ast}= 2 Id : T(X)\to T(X)$, and also $\phi_{\ast}\circ \phi^{\ast}= 2 Id $. Further Roitman's result~\cite{roitman} that for any smooth projective surface $Y$, $T(Y)$ is torsion free, implies that $\phi^{\ast}: T(Km(X))\to T(X)$ is actually an isomorphism. Also the automorphism $g$ commutes with the rational map $\phi.$ Therefore have the commuting diagram:
   \begin{equation}
\label{V}
\begin{tikzpicture}
[back line/.style={densely dotted},cross line/.style={loosely dotted}]
\matrix (m) [matrix of math nodes, row sep=2 em,column sep=2.em, text height=1.5ex, text depth=0.50ex]
 {
  T(Km(X)) & T(Km(X)) \\
 T(X) & T(X) \\
 };
\path[->,font=\scriptsize]
  (m-1-2) edge node[auto] {$\phi^{\ast}$} (m-2-2)
  (m-1-1) edge node [auto] {$\phi^{\ast}$} (m-2-1)
  (m-1-1) edge node [auto] {$g_{\ast}$} (m-1-2)

(m-2-1) edge node [auto] {$f_{\ast}$} (m-2-2);
\end{tikzpicture}
\end{equation}
where the vertical arrows are isomorphisms. \\
Now let us assume that $f$ satisfies the assumptions of the Corollary~\ref{mc}, then by the conclusion $f_{\ast}:T(X)\to T(X)$ is Id. Hence by above discussion $g_{\ast}:T(Km(X))\to T(Km(X))$ is Id. Further, since the Albanese variety $Alb(Km(X))$ associated to the Kummer surface is 0, the Albanese morphism $alb: {CH_0(Km(X))}_0\to Alb(Km(X))$ is 0. Thus,
\begin{center}
$T(Km(X))= {CH_0(Km(X))}_0\subset CH_0(Km(X))$. 
\end{center}
One can define a splitting of the inclusion above compatible with the morphism $g_{\ast}.$ Fix a point $o\in Km(X)$ lying over the image of the origin of $X$ in the quotient variety $X/G$. Define 
\begin{center}
$\lambda: CH_0(Km(X))\to {CH_0(Km(X))}_0$ by $\alpha\mapsto \alpha- deg(\alpha)[o]$, 
\end{center}
where $[o]$ is the equivalence class associated to the point $o$ in $CH_0(Km(X))$ (this class $[o]$ is independent of the point chosen in the fiber over the image of the origin of $X$ in the quotient variety $X/G$, since the fiber is rationally connected). One can observe that the following diagram
   \begin{equation}
\label{5}
\begin{tikzpicture}
[back line/.style={densely dotted},cross line/.style={loosely dotted}]
\matrix (m) [matrix of math nodes, row sep=2 em,column sep=2.em, text height=1.5ex, text depth=0.50ex]
 {
  CH_0(Km(X)) & CH_0(Km(X))_0 & CH_0(Km(X)) \\
 CH_0(Km(X)) & CH_0(Km(X))_0 & CH_0(Km(X))\\
 };
\path[->,font=\scriptsize]
  (m-1-2) edge node[auto] {$g_{\ast}$} (m-2-2)
              edge [draw=none]
                                    node [sloped, auto=false,
                                     allow upside down]  {$\subset$} (m-1-3)
  (m-1-1) edge node [auto] {$g_{\ast}$} (m-2-1)
  (m-1-1) edge node [auto] {$\lambda$} (m-1-2)
(m-1-3) edge node [auto] {$g_{\ast}$} (m-2-3)
(m-2-1) edge node [auto] {$\lambda$} (m-2-2)
  (m-2-2)  edge [draw=none]
                                    node [sloped, auto=false,
                                     allow upside down]  {$\subset$} (m-2-3);
\end{tikzpicture}
\end{equation}
commutes. Now let $\alpha\in CH_0(Km(X))$. Thus, $g_{\ast}(\lambda(\alpha))=\lambda(g_{\ast}(\alpha))$. Since $g_{\ast}$ is Identity on $CH_0(Km(X))_0$, hence on the image of $\lambda$, $\lambda(\alpha)= g_{\ast}(\alpha)-deg(g_{\ast}(\alpha))[o].$ Thus $\alpha-deg(\alpha)[o]=g_{\ast}(\alpha)-deg(g_{\ast}(\alpha))[o]$. Hence $\alpha=g_{\ast}(\alpha).$

Thus, $g_{\ast}$ is the identity on $CH_0(Km(X)).$

So from above discussion, we have proved the following:
\begin{theorem}
\label{km}
Let $g$ be an automorphism of the Kummer surface $Km(X)$ associated to the abelian surface $X$, which is induced from an automorphism $f$ of $X$ as abelian variety. Further assume that $f$ is a symplectic automorphism of $X$. Then $g$ acts as the identity on $CH_0(Km(X))$.
\end{theorem}
\begin{remark}
In view of the Theorem~\ref{hv}, the Theorem~\ref{km} deals with automorphisms of the Kummer surface which are possibly of infinite order. One can construct Kummer surfaces with automorphisms of infinite order, for example for the Kummer surface associated with the abelian surface in the Example~\ref{foreg} above. We can get a similar result for generalized Kummer varieties $K_n(X)$ associated to the abelian surface $X$, where for $n=1$, $K_1(X)=Km(X).$
\end{remark}

\paragraph{Generalized Kummer varieties}
First we recall (see~\cite{b83}) the construction and properties of generalized Kummer varieties $K_n$, and the result of~\cite{lin}, which gives a Beauville type decomposition on $CH_0(K_n)$ in terms of that of $CH_0(X)$, for a complex abelian surface $X$.

Let $X_0^{n+1}$ = kernel of the sum map $\mu:X^{n+1}\to X$. The symmetric group $\mathfrak{S}_{n+1}$ acts on $X^{n+1}$, hence also on $X_0^{n+1}$. Have the quotient map $\rho: X^{n+1}\to X^{n+1}/\mathfrak{S}_{n+1}$ with restriction to $X_0^{n+1}$ induces its own quotient map $q:X_0^{n+1} \to K_{(n)}:=X_0^{n+1}/\mathfrak{S}_{n+1}$. On the other hand $\mu$ induces the sum map $\tilde{s}: X^{n+1}/\mathfrak{S}_{n+1}\to X$. One can check that $K_{(n)}= {\tilde{s}}^{-1}(0).$ \\
Let $\Delta \subseteq X^{n+1}$ be the closed subscheme consisting of points in $X^{n+1}$ with atleast 2 equal coordinates. Write $D:=\rho(\Delta)\subseteq X^{n+1}/\mathfrak{S}_{n+1}$. Let $\tilde{\nu}:X^{[n+1]}\to X^{n+1}/\mathfrak{S}_{n+1}$ be the Hilbert-Chow morphism, with $X^{[n+1]}$ as the Hilbert scheme of $X$ of length $n+1$ closed subschemes of $X$. This morphism is a map of resolution of singularities with the exceptional divisor $E:=\tilde{\nu}^{-1}(D)$. 
Define $K_n:=(\tilde{s}\circ\tilde{\nu})^{-1}(0).$ \\
Now the restriction of the Hilbert-Chow morphism $\tilde{\nu}$ to $K_n$ gives $\nu: K_n\to K_{(n)}$ which is a map of resolution of singularities and the exceptional divisor $E_0:=E\cap K_n$. Set $D_0=D\cap K_{(n)}$ and $\Delta_0=\Delta\cap X_0^{{n+1}}$. Then one has $\rho(\Delta_0)=D_0.$ We have following diagram of spaces:
\begin{equation}
\begin{tikzpicture}
[back line/.style={densely dotted},cross line/.style={loosely dotted}]
\matrix (m) [matrix of math nodes, row sep=.8 em,column sep=1.5em, text height=1.5ex, text depth=0.50ex]
{E_0 &      & E &   \\
        & K_n &    & X^{[n+1]}  \\
 D_0 &       &  D &     \\
        &  K_{(n)} &   & X^{n+1}/\mathfrak{S}_{n+1} \\
        &             &   X &    \\
  };
 \path[right hook->, font=\scriptsize]
  (m-2-2) edge node [auto]{} (m-2-4)
   (m-4-2) edge node [auto] {} (m-4-4);

\path[->, font=\scriptsize]
 (m-1-1) edge node [auto] {} (m-1-3)
             edge [draw=none]
                                    node [sloped, auto=false,
                                     allow upside down] {$\subseteq$} (m-2-2)
  edge node [auto] {} (m-3-1)
  (m-1-3) 
              edge [draw=none]
                                    node [sloped, auto=false,
                                     allow upside down] {$\subseteq$} (m-2-4)
           edge node [auto] {} (m-3-3)
  (m-2-2) 
           edge node [above left]{$\nu$} (m-4-2) 
  (m-2-4) edge node [above right] {$\tilde{\nu}$} (m-4-4) 
  (m-3-1) edge node [auto] {} (m-3-3) 
      edge [draw=none]
                                    node [sloped, auto=false,
                                     allow upside down] {$\subseteq$}(m-4-2)
 (m-3-3) edge [draw=none]
                                    node [sloped, auto=false,
                                     allow upside down] {$\subseteq$} (m-4-4)
    (m-4-4) edge node [below right] {$\tilde{s}$} (m-5-3)     
  (m-4-2) 
             edge node [below left] {$0$} (m-5-3);
\end{tikzpicture}
\label{2}
\end{equation}

\begin{remark}
$K_n$ constructed above is called a \textit{generalized Kummer variety} associated to the abelian surface $X$. For $n=1$ one can recover the Kummer surface associated to $X$ by above construction. $K_n$ is an irreducible, projective, hyperk\"ahler variety of dimension $2n$. 
\end{remark}

Since the multiplication morphism $m:X_0^{n+1} \to X_0^{n+1}$ commutes with the action of $\mathfrak{S}_{n+1}$, for each $m\in\mathbf{Z}$,
\begin{center}
$\displaystyle q^{\ast}: CH_0(K_{(n)})\xrightarrow{\sim} \Big(\bigoplus_{s=0}^{2n} CH_0(X_0^{n+1})_s\Big)^{\mathfrak{S}_{n+1}}=\bigoplus_{s=0}^{2n} CH_0(X_0^{n+1})_s^{\mathfrak{S}_{n+1}}$
\end{center}
 Since $q_{\ast}q^{\ast}: CH_0(K_{(n)})\to CH_0(K_{(n)})$ is given by $(n+1)!$ (\cite[Lemma 1.7.6]{fulton}) and the Chow groups are with rational coefficients, we get that $q_{\ast}$ is bijective.
Thus,
we get
\begin{center}
$\displaystyle q_{\ast}: \bigoplus_{s=0}^{2n} CH_0(X_0^{n+1})_s^{\mathfrak{S}_{n+1}}\to CH_0(K_{(n)})$
\end{center}
is bijective.
We obtain a following decomposition 
\begin{center}
$\displaystyle CH_0(K_{(n)})=\bigoplus_{s=0}^{2n} CH_0(K_{(n)})_s$
\end{center}
where $CH_0(K_{(n)})_s:= q_{\ast}CH_0(X_0^{n+1})_s^{\mathfrak{S}_{n+1}}$.
Now the Hilbert-Chow morphism $\nu: K_n\to K_{(n)}$ given by the desingularization (Hilbert-Chow morphism has fibers rationally connected) induces an isomorphism $\nu_{\ast}:CH_0(K_n)\to CH_0(K_{(n)})$ which induces a decomposition on $CH_0(K_n)$
\begin{equation}
\label{0cyclesonkummer}
\displaystyle CH_0(K_{n})=\bigoplus_{s=0}^{2n} CH_0(K_{n})_s,
\end{equation}
 where $CH_0(K_{n})_s:=\nu_{\ast}^{-1}(CH_0(K_{(n)})_s)$.

\paragraph{$Aut(X)$ and $Aut(K_n)$:}
 Next we want to compare automorphisms of $X$ and $K_n$. 
 Let $f\in Aut(X)$ as abelian variety. We assign to $f$ a unique $f_{n}\in Aut(K_n)$ as follows:
  \begin{equation}
\begin{tikzpicture}
[back line/.style={densely dotted},cross line/.style={loosely dotted}]
\matrix (m) [matrix of math nodes, row sep=2 em,column sep=2.em, text height=1.5ex, text depth=0.50ex]
{K_n & K_{(n)} & X^{n+1}_0 & X^{n+1} & X\\
  K_n & K_{(n)} & X^{n+1}_0 & X^{n+1} & X\\
  };

\path[->,font=\scriptsize]
 (m-1-1) edge node [auto] {$\nu$} (m-1-2)
             edge node [auto] {$f_{n}$} (m-2-1)
  (m-1-2) edge node [auto] {$f_{(n)}$} (m-2-2)
   (m-1-3) edge node [auto] {$q$} (m-1-2)
  (m-1-3) edge node [auto] {} (m-1-4)
              edge node [auto] {$f_0$} (m-2-3)
   (m-1-4) edge node [auto] {$\mu$} (m-1-5)
           edge node [auto] {$f^{\times(n+1)}$} (m-2-4)
    (m-1-5) edge node [auto] {$f$} (m-2-5)
   (m-2-1) edge node [auto] {$\nu$} (m-2-2)
  (m-2-3) edge node [auto]{$q$} (m-2-2) 
   (m-2-3) edge node [auto] {} (m-2-4)
  (m-2-4) edge node [auto] {$\mu$} (m-2-5) ;
\end{tikzpicture}
\label{3}
\end{equation}


Now as we did for the Kummer surface,
\begin{theorem}
\label{gkm}
Let $f_n$ be the automorphism of $K_n$ induced from an automorphism $f$ of $X$. Suppose $f$ is symplectic on $X$. Then $f_n$ is symplectic on $K_n$ and $f_{n}$ acts as the identity on $CH_0(K_n)_{2n}.$
\end{theorem}
\begin{proof}
\begin{itemize}
\item $f_n$ is symplectic: We have $H^2(K_n, \mathbf{C})\simeq H^2(X,\mathbf{C})\oplus \mathbf{C}[E_0].$ 
Further $H^{2,0}(K_n)\simeq H^{2,0}(X)$, which is functorial. Hence if $f$ acts as the identity on $H^{0}(X, \Omega_X^2)$, then $f_n$ acts as the identity on $H^{0}(K_n,\Omega_{K_n}^2)$. Thus, $f_n$ is symplectic. 
\item By the Corollary~\ref{mc}, $f$ acts as Id on $CH_0(X)_2.$ Next the map $f^{\times(n+1)}:X^{n+1}\to X^{n+1}$ induces an automorphism $f_0$ of $X_0^{n+1}$. Consider the action of $f_0$ on $CH_0(X_0^{n+1})^{\mathfrak{S}_{n+1}}$.

\end{itemize}

We will identify $X_0^{n+1}$ with $X^n$ via
\begin{center}
$\displaystyle X_0^{n+1}\to X^n : (z_1,z_2,\cdots,z_n,-\sum_{j=1}^nz_j)\mapsto (z_1,z_2,\cdots,z_n).$
\end{center}
$\mathfrak{S}_{n+1}$ acts on $X^n$, via the action on $X_0^{n+1}$ and the above isomorphism. Write $\mathfrak{S}_m$ as acting on the set $\{1,2,\cdots, m\}$, and view $\mathfrak{S}_{n}\subset\mathfrak{S}_{n+1}$. \\
Action of $\mathfrak{S}_{n+1}$ on $X^n: \mathfrak{S}_n$ acts on $X^n$ by permuting the co-ordinates. Write $t_i\in \mathfrak{S}_{n+1}$ for the transposition $(i, n+1)$ which permutes the $i$ with $n+1$, for $1\leq i\leq n$; the action of $t_i$ on $X^n$ defined by
\begin{center}
$\displaystyle t_i\cdot(z_1,z_2,\cdots,z_n)\mapsto (z_1,z_2,\cdots,z_{i-1},-\sum_{j=1}^nz_j,z_{i+1},\cdots,z_n).$
\end{center}
   
     First we identify cycles in the $\mathfrak{S}_{n+1}$-invariant eigenspaces $CH_0(X^n)_r^{\mathfrak{S}_{n+1}}$ of the Beauville decomposition for $X^n$ in terms of the cycles in $CH_0(X)_{s}$ for $s=0, 2$. 
   We have the following surjective homomorphism given by the external product of 0-cycles
   \begin{center}
$\underbrace{CH_0(X)\otimes CH_0(X)\otimes\cdots\otimes CH_0(X)}_{\text{$n$ times}}\xrightarrow{p_1^{\ast}(-)\cdots p_n^{\ast}(-)} CH_0(X^n)$,
   \end{center}
   where $p_i: X^n\to X$ is the projection onto the $i^{th}$ component and for 0-cycles $\alpha_1,\cdots,\alpha_n$, 
   \begin{center}
$p_1^{\ast}(\alpha_1)\cdots p_n^{\ast}(\alpha_n)$ 
   \end{center}
denotes the intersection product of the cycles $p_1^{\ast}(\alpha_1),\cdots, p_n^{\ast}(\alpha_n)$ on $X^n$.
We compose this map with the symmetrizer map to get
\begin{center}
$\underbrace{CH_0(X)\otimes CH_0(X)\otimes\cdots\otimes CH_0(X)}_{\text{$n$ times}}\xrightarrow{p_1^{\ast}(-)\cdots p_n^{\ast}(-)} CH_0(X^n)\xrightarrow{\sum_{\sigma\in\mathfrak{S}_{n+1}}\sigma^{\ast}(-)}CH_0(X^n)^{\mathfrak{S}_{n+1}}.$
\end{center}
By Beauville's decomposition (as in Theorem~\ref{bdec}) $CH_0(X)=\bigoplus_{s=0}^2 CH_0(X)_{s}$ ; $CH_0(X^n)=\bigoplus_{r=0}^{2n} CH_0(X^n)_{r}$,  and composing the above map, for every $r$ such that $0\leq r\leq 2n$, we have the surjective homomorphisms $\Psi_{r}$, 
\begin{center}
$\displaystyle \Psi_r:\bigoplus_{\substack{s_1+s_2+\cdots+s_n=r\\ s_j=0, 1, 2}} CH_0(X)_{s_1}\otimes CH_0(X)_{s_2}\otimes\cdots\otimes CH_0(X)_{s_n}\to CH_0(X^n)_r^{\mathfrak{S}_{n+1}}$.
\end{center}
Now one can easily observe that for $r=2n$, the homomorphism 
   
   \begin{center}

    $\Psi_{2n}:\underbrace{CH_0(X)_{2}\otimes CH_0(X)_{2}\otimes\cdots\otimes CH_0(X)_{2}}_{\text{$n$ times}}\to CH_0(X^n)_{2n}^{\mathfrak{S}_{n+1}}$
   \end{center}
is surjective.

 \paragraph{Back to proof of Theorem~\ref{gkm}:}
Since $f_0$ acts on $CH_0(X_0^{n+1})^{\mathfrak{S}_{n+1}}$ via the action of $f$ on 
\begin{center}
$\underbrace{CH_0(X)_{2}\otimes CH_0(X)_{2}\otimes\cdots\otimes CH_0(X)_{2}}_{\text{$n$ times}}$
\end{center}

 which is identity, as $f$ acts as the identity on $CH_0(X)_{2}$.
 
 Hence $f_0$ acts on $CH_0(X^n)_{2n}^{\mathfrak{S}_{n+1}}$ as the identity. Thus $f_n$ acts on $CH_0(K_n)_{2n}$ by the identity. 
 
 \end{proof}

\begin{remark}
The Theorem~\ref{gkm} can be compared with
a general conjecture made in~\cite[Conjecture 0.3]{fu}, finite order symplectic automorphisms of irreducible hyperk\"ahler varieties act as the identity on $CH_0$. Further, the automorphism $f_n$ of $K_n$ above could possibly be of infinite order.    
\end{remark}
\end{section}

\paragraph {Acknowledgement} 
I would like to express my gratitude to my thesis advisor Prof. V. Srinivas for introducing me to the subject and suggesting the problem as well as constant guidance and encouragement that led to this paper. I would also like to thank Prof. N. Fakhruddin for bringing~\cite{o} to my attention and subsequent suggestions during the work. I would like to thank both for pointing out errors in the earlier version, and suggestions which immensely improved the exposition. I would like to thank D. Huybrechts and H.-Y. Lin for pointing out error in the earlier version.   
I am supported by SPM fellowship funded by CSIR, India (SPM-07/858(0139)/2012).

\bibliographystyle{} 
\medskip
\medskip
\medskip

Rakesh Pawar, \textsc{School of Mathematics,
  Tata Institute of Fundamental Research, Mumbai,
  Mumbai - 400 005, India.
  }\par\nopagebreak
  \textit{E-mail address}: \texttt{Email: rpawar@math.tifr.res.in}

\end{document}